\documentclass[12pt]{article}
\usepackage[T1]{fontenc}
\usepackage[utf8]{inputenc}
\usepackage{lmodern}
\usepackage{amsmath,amssymb,amsfonts,amsthm,mathtools,mathrsfs,bm}
\usepackage[a4paper,margin=1.05in]{geometry}
\usepackage{microtype}
\usepackage{enumitem}
\usepackage{cite}
\usepackage{xcolor}
\usepackage{graphicx}
\usepackage{float}
\usepackage[font=small,labelfont=bf]{caption}
\usepackage{hyperref}
\allowdisplaybreaks[3]
\numberwithin{equation}{section}
\hypersetup{colorlinks=true,linkcolor=blue,citecolor=blue,urlcolor=blue,
filecolor=blue,anchorcolor=blue,menucolor=blue,linktoc=all,
bookmarksopen=true,bookmarksnumbered=true,
pdftitle={Wulff Rigidity for Overdetermined Anisotropic k-Hessian Equations in General Smooth Domains},
pdfauthor={Lichun Liang and Yanjun Liu},
pdfsubject={Anisotropic k-Hessian equations and Wulff rigidity},
pdfkeywords={anisotropic k-Hessian equation, Wulff shape, overdetermined boundary condition, Pohozaev identity}}
\newtheorem{theorem}{Theorem}[section]

\newtheorem{lemma}[theorem]{Lemma}
\newtheorem{corollary}[theorem]{Corollary}
\theoremstyle{definition}

\theoremstyle{remark}
\newtheorem{remark}[theorem]{Remark}
\newcommand{\R}{\mathbb{R}}
\newcommand{\dd}{\,\mathrm{d}}
\newcommand{\Wulff}{\mathcal W}
\title{\bfseries Wulff Rigidity for Overdetermined\\
Anisotropic $k$-Hessian Equations}
\author{\textsc{Lichun Liang}\quad and \quad
\textsc{Yanjun Liu}\thanks{Y. Liu is supported by NSFC grants 12201089 and 12371051.}}
\date{}
\begin{document}
\maketitle

\begin{abstract}
We establish Wulff rigidity for overdetermined anisotropic
$k$-Hessian equations associated with a smooth and strongly convex  Finsler gauge.
The solutions are subject to homogeneous Dirichlet boundary and
a prescribed anisotropic  gradient condition on a smooth domain.
We prove that the domain must be a Wulff shape.
The same rigidity holds under a prescribed positive weighted Neumann-type condition.
These results extend the Wulff rigidity  for the endpoint cases $k=1$ (Cianchi--Salani, \emph{Math. Ann.}, 2009; Wang--Xia \emph{Arch. Ration. Mech. Anal.}, 2011) and $k=n$ ( Cianchi--Salani, \emph{J. Funct. Anal.}, 2023) to the intermediate cases $2\le k<n$.
\end{abstract}



\section{Introduction and main results}
\noindent

Overdetermined  problems have a long history and reveal how boundary data constrain the geometry of a domain. For the classical Poisson torsion problem, corresponding to
the isotropic case $k=1$, Serrin~\cite{b18} used the moving-plane method to prove that zero Dirichlet data and a constant normal derivative force a bounded and smooth domain to be a ball. Weinberger~\cite{b25} subsequently gave a shorter proof of the
same rigidity result based on a $P$-function, the maximum principle and an integral identity. 
For general isotropic $k$-Hessian equations, Brandolini, Nitsch, Salani, and Trombetti~\cite{b5} established ball rigidity for all $1\le k\le n$ on bounded $C^2$ domains under zero Dirichlet data and a constant positive normal derivative. Related results concern mirror symmetry for Hessian overdetermined problems~\cite{b21}, Hessian quotient equations and their curvature counterparts~\cite{b12} and exterior $k$-Hessian problems~\cite{b22}. Connections with higher order mean curvatures and quantitative
stability were further studied in~\cite{b13}.

In the anisotropic setting, Wulff shapes replace Euclidean balls.
For $k=1$ and reversible anisotropy, Cianchi and Salani~\cite{b7} established Wulff rigidity under zero Dirichlet data and constant positive anisotropic boundary gradient. Wang and Xia~\cite{b23} proved the same rigidity on smooth domains. For $k=n$, Cianchi and Salani~\cite{b8} established Wulff rigidity for convex weak Finsler Monge--Amp\`ere solutions on bounded convex domains, under reversible anisotropy and a prescribed gradient image.
Related work includes anisotropic symmetrization~\cite{b1,b20}, maximum and comparison principles for the Finsler Laplacian~\cite{b10}, integral methods for degenerate
anisotropic equations~\cite{b2}, capacitary characterizations  for convex domains~\cite{b3,b4}, mixed-volumesymmetrization~\cite{b9} and curvature-based Wulff
characterizations and stability~\cite{b14,b15,b17}.
More recently, for $k=1$, Figalli and Zhang~\cite{b11} established anisotropic Serrin rigidity for weak solutions on Lipschitz and more general rough domains under
geometric measure assumptions.

The elliptic theory of $k$-Hessian equations on admissible
branches is well developed~\cite{b6,b19,b24}. Their connection with curvature geometry~\cite{b16} makes Newton tensors and curvature integrals natural tools for
overdetermined problems. 

Motivated by these developments, we establish Wulff rigidity for the anisotropic $k$-Hessian problem for $1\le k<n$ on bounded $C^3$ domains.
The same conclusion holds under the prescribed positive weighted Neumann-type condition, which ensures $k$-admissibility automatically.

For  $1\le k\le n$, we let $S_k(A)$ denote the $k$-th elementary
symmetric polynomial of the eigenvalues of an $n\times n$ matrix $A$. 
Throughout the paper, we assume that $F\in C^3(\R^n\setminus\{0\})$ is a strongly convex Finsler gauge, i.e., 

\textbf{(1). (Positivity)} \[
F(0)=0,\qquad F(\xi)>0\quad(\xi\ne0);\]

\textbf{(2). ($1$-homogeneous)} 
$$F(t \xi)=tF(\xi),\ \ \ t>0 \ \ \ \mbox{and }\ \ \ \xi\in \mathbb{R}^n;$$

\textbf{(3). (Strongly convex)}
$$\frac{1}{2}D_\xi^2 F^2(\xi)>0,\ \ \ \xi\in \R^n\setminus\{0\}.$$
We set
\begin{equation}\label{eq:E-definition}
E(\xi)=\frac12F(\xi)^2
\end{equation}
and therefore define, almost everywhere in the weak sense,
\begin{equation}\label{eq:anisotropic-hessian}
A_F[u]:=D_x\bigl(D_\xi E(Du)\bigr)
\qquad \mbox{and}\qquad
S_{k,F}[u]:=S_k\bigl(A_F[u]\bigr).
\end{equation}
This is the anisotropic Hessian operator introduced in the framework of mixed-volume symmetrization~\cite{b9}. For $k=1$ it is the Finsler Laplacian; for $k=n$ it is a Finsler Monge--Amp\`ere operator.

The polar gauge and the unit Wulff shape are
\begin{equation}\label{eq:polar-norm}
F_0(x)=\sup_{\xi\ne0}\frac{x\cdot\xi}{F(\xi)}
\qquad \mbox{and}\qquad
\Wulff=\{x\in\R^n:F_0(x)<1\}.
\end{equation}
For any gauge $G$ and $r>0$, we write $B_G(r)=\{\xi:G(\xi)<r\}$, so $\Wulff=B_{F_0}(1)$. Consider the normalized model profile 
\begin{equation}\label{eq:model-unit}
u_0(x)=\frac{F_0(x)^2-1}{2},
\end{equation}
which satisfies
\[
D_\xi E(Du_0(x))=x,\qquad A_F[u_0]=I,\qquad S_{k,F}[u_0]=\binom nk.
\] 
The normalized model profile need not be $C^2$ at its minimum.
To include this profile, we assume that the critical set
$\mathcal C_u=\{x\in\Omega:Du(x)=0\}$ is finite and that
\begin{equation}\label{eq:regularity-class}
u\in C^{1,1}(\overline\Omega)
\cap C^3(\overline\Omega\setminus\mathcal C_u).
\end{equation}
For $1\le k\le n$, the G\aa rding cone $\Gamma_k$ is 
\[
\Gamma_k=\bigl\{\lambda\in\R^n:S_j(\lambda)>0,\ 1\le j\le k\bigr\}.
\]
When $Du\ne0$, we know that $A_F[u]=D_\xi^2E(Du)D^2u$ is similar to a symmetric matrix. We say that  $u$ is $k$-admissible if $\lambda(A_F[u])\in\Gamma_k$ almost everywhere.

Our first result establishes Wulff rigidity under zero Dirichlet data and a constant positive anisotropic boundary gradient.
\begin{theorem}\label{thm:main}
Let  $F\in C^3(\R^n\setminus\{0\})$ be a strongly convex  Finsler gauge and  $\Omega\subset\R^n$ be a bounded domain with $C^3$ boundary. For $1\le k<n$, assume that $u$ satisfies \eqref{eq:regularity-class} with finite  critical set and is a $k$-admissible solution of
\begin{equation}\label{eq:main-problem}
\begin{cases}
S_{k,F}[u]=\binom nk & \text{in }\Omega,\\
u=0 & \text{on }\partial\Omega,\\
F(Du)=c & \text{on }\partial\Omega
\end{cases}
\end{equation}
for some $c>0$. Then there exists $x_0\in\R^n$ such that
\begin{equation}\label{eq:main-conclusion}
\Omega=x_0+c\Wulff \qquad \mbox{and}\qquad
u(x)=\frac{F_0(x-x_0)^2-c^2}{2}\quad\text{for }x\in\Omega.
\end{equation}
\end{theorem}

Figure~\ref{fig:wulff} illustrates the model pair in \eqref{eq:main-conclusion}.

\begin{figure}[H]
\centering
\includegraphics[width=332bp]{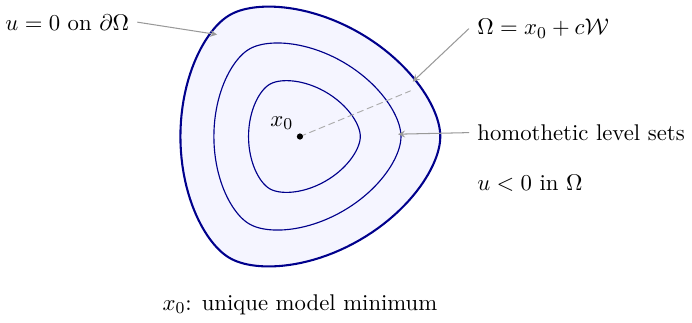}
\caption{A two-dimensional illustration of the Wulff domain in Theorem~\ref{thm:main} and the level
sets of its polar quadratic profile. The point $x_0$ is the center of dilation; the gauge need not be
reversible.}
\label{fig:wulff}
\end{figure}

Let $\nu$ denote the outward Euclidean unit normal and set $u_\nu=Du\cdot\nu$. In Theorem~\ref{thm:main}, $k$-admissibility and the Newton--Maclaurin inequality give
$S_{1,F}[u]\ge n$, so the weak and strong maximum principles yield $u<0$ in $\Omega$. Since $u=0$ on $\partial\Omega$, we have $Du=u_\nu\nu$ and $u_\nu\ge0$ . Moreover, by the condition $F(Du)=c>0$ and the positive homogeneity of $F$, we know that 
 $u_\nu>0$ and 
\[
F(Du)=u_\nu F(\nu)=c\qquad\text{on }\partial\Omega.
\]
This sign determines the boundary orientation for non-reversible gauges, for which $F(\nu)$ and $F(-\nu)$ may differ. For $k=1$, admissibility follows directly from
the equation, as in the anisotropic Serrin problem~\cite{b7,b23}.

Our next result replaces the boundary-gradient condition
by the positive weighted conormal condition
\begin{equation}\label{eq:N-flux}
F(Du)D_\xi F(Du)\cdot\nu=cF(\nu)
\quad\text{on }\partial\Omega,\qquad c>0.
\end{equation}
Using  $Du=u_\nu\nu$ and the Euler's identity of $E$ yields 
\[ u_\nu D_\xi E(Du)\cdot\nu =Du\cdot D_\xi E(Du) =F(Du)^2.
\]
Together this with \eqref{eq:N-flux}, we have $u_\nu>0$. Then the positive homogeneity of $F$ yields
\begin{equation}\label{eq:N-bdry}
F(Du)D_\xi F(Du)\cdot\nu
=u_\nu F(\nu)D_\xi F(\nu)\cdot\nu
=u_\nu F(\nu)^2.
\end{equation}
Thus, under zero Dirichlet data and $c>0$, we see that 
\begin{equation}\label{eq:N-equiv}
\begin{aligned}
F(Du)D_\xi F(Du)\cdot\nu=cF(\nu)
&\quad\Longleftrightarrow\quad
u_\nu=\frac{c}{F(\nu)}>0\\
&\quad\Longleftrightarrow\quad
F(Du)=c\ \text{and}\ u_\nu>0.
\end{aligned}
\end{equation}

 We refer to \eqref{eq:N-flux} as an \emph{anisotropic Neumann-type} condition. In particular, when $F(\xi)=|\xi|$, it reduces to $u_\nu=c$.

\begin{theorem}\label{thm:neumann}
Let  $F\in C^3(\R^n\setminus\{0\})$ be a strongly convex  Finsler gauge and  $\Omega\subset\R^n$ be a bounded  domain with $C^3$ boundary. For $1\le k<n$, assume  that $u$ satisfies \eqref{eq:regularity-class} with finite  critical set and is a  solution of
\begin{equation}\label{eq:N-pb}
\begin{cases}
S_{k,F}[u]=\binom nk & \text{in }\Omega,\\
u=0 & \text{on }\partial\Omega,\\
F(Du)D_\xi F(Du)\cdot\nu=cF(\nu) & \text{on }\partial\Omega.
\end{cases}
\end{equation}
for some $c>0$.  Then $u$ is $k$-admissible, $u<0$ in $\Omega$ and
\eqref{eq:main-conclusion} holds for some $x_0\in\R^n$.
\end{theorem}

In Theorem~\ref{thm:main}, $k$-admissibility yields
$u_\nu>0$ on $\partial\Omega$; in Theorem~\ref{thm:neumann}, positive conormal flux yields
$k$-admissibility. Both results prove $u<0$ in $\Omega$.
The condition $c>0$ is essential for selecting the admissible branch, as illustrated by Remark~\ref{rem:N-sign}.

By the positive homogeneity, both results extend by scaling to any positive constant right-hand side.

\begin{corollary}\label{cor:general-scaling}
Let $1\le k<n$, $\mu,c>0$ and let $F$ and $\Omega$ be as in Theorem~\ref{thm:main}.
Suppose that  $u$ satisfies \eqref{eq:regularity-class} with finite  critical set and is a solution of
\[
\begin{cases}
S_{k,F}[u]=\mu & \text{in }\Omega,\\
u=0 & \text{on }\partial\Omega.
\end{cases}
\]
Assume either that $u$ is $k$-admissible and $F(Du)=c$ on $\partial\Omega$ or that
\eqref{eq:N-flux} holds. 
Then $u$ is $k$-admissible and $u<0$ in $\Omega$. Moreover, we have
\begin{equation}\label{eq:general-profile}
\Omega=x_0+R\Wulff \qquad \mbox{and}\qquad
u(x)=\frac{\alpha}{2}
\bigl(F_0(x-x_0)^2-R^2\bigr)
\end{equation}
for some $x_0\in\R^n$, where 
\begin{equation}\label{eq:alpha-R}
\alpha=\left(\frac{\mu}{\binom nk}\right)^{1/k}
\qquad \mbox{and}\qquad R=\frac{c}{\alpha}.
\end{equation}
\end{corollary}

For uniformly convex solutions, prescribing the gradient image gives a second boundary value formulation.

\begin{corollary}\label{cor:gradient-image}
Let $1\le k<n$ and let $F\in C^3(\R^n\setminus\{0\})$ be a strongly convex  Finsler gauge and  $\Omega\subset\R^n$ be a bounded convex domain with $C^3$ boundary.
Suppose that $u$ satisfies \eqref{eq:regularity-class} with finite critical set
and is uniformly convex in the sense that $D^2u\ge\theta I$ almost everywhere in $\Omega$
for some $\theta>0$. If, for some $c>0$, $u$ satisfies
\begin{equation}\label{eq:gradient-image-problem}
\begin{cases}
S_{k,F}[u]=\binom nk & \text{in }\Omega,\\
u=0 & \text{on }\partial\Omega,\\
Du(\Omega)=B_F(c).
\end{cases}
\end{equation}
then \eqref{eq:main-conclusion} holds for some $x_0\in\R^n$.
\end{corollary}

The integral argument yields an exact identity with three
nonnegative deficits whose vanishing forces rigidity.
On general domains, the geometric terms involve the signed
anisotropic curvature integrals $\mathcal W_{j,F}$ defined
in \eqref{eq:signed-curvature-functional}.
The identity follows directly from the integral calculations,
before the equality case is invoked.

\begin{theorem}\label{thm:deficit}
Let  $F\in C^3(\R^n\setminus\{0\})$ be a strongly convex  Finsler gauge and  $\Omega\subset\R^n$ be a bounded domain with $C^3$ boundary. For $1\le k<n$, assume that $u$ satisfies \eqref{eq:regularity-class} with finite  critical set and is a $k$-admissible solution of \eqref{eq:main-problem} with $c=1$. 
Then  
\begin{equation}\label{eq:deficit-identity-intro}
\begin{aligned}
0={}&(k+1)\int_\Omega(-u)
\left(\binom{n}{k+1}-S_{k+1,F}[u]\right)\dd x\\
&+\frac{k}{2}\binom{n}{k}
\left(\mathcal W_{k-1,F}(\Omega)-|\Omega|\right)\\
&+\frac{k+2}{2}\binom{n}{k}
\left(
\int_\Omega F(Du)^2\dd x
-\frac{n}{n+2}\mathcal W_{k-1,F}(\Omega)
\right).
\end{aligned}
\end{equation}
Each term on the right-hand side is nonnegative, so all
three vanish. Moreover, we have
\begin{equation}\label{eq:all-deficits-zero}
\begin{aligned}
S_{k+1,F}[u]
&=\binom{n}{k+1}
\quad\text{a.e. in }\Omega,\\
\mathcal W_{k-1,F}(\Omega)
&=|\Omega|,\\
\int_\Omega F(Du)^2\dd x
&=\frac{n}{n+2}\mathcal W_{k-1,F}(\Omega).
\end{aligned}
\end{equation}
For $k=1$, the equality $\mathcal W_{0,F}(\Omega)=|\Omega|$ holds by definition.
\end{theorem}

The vanishing of the weighted Newton deficit characterizes
the Wulff domain and its associated solution without
prescribing $F(Du)$ on $\partial\Omega$.

\begin{theorem}\label{thm:equivalence}
Let $1\le k<n$, let $F\in C^3(\R^n\setminus\{0\})$ be a strongly convex Finsler gauge, and let $\Omega\subset\R^n$ be a bounded domain
with $C^3$ boundary. Assume that $u$ satisfies \eqref{eq:regularity-class} with finite  critical set and is a $k$-admissible solution of
\[
\begin{cases}
S_{k,F}[u]=\binom{n}{k}
& \text{in }\Omega,\\
u=0
& \text{on }\partial\Omega.
\end{cases}
\]
Then the following statements are equivalent:
\begin{enumerate}[label=\textup{(\roman*)},leftmargin=2.8em]
\item
$F(Du)$ is a positive constant on $\partial\Omega$;

\item
the weighted Newton deficit vanishes:
\[
\int_\Omega(-u)
\left(\binom{n}{k+1}-S_{k+1,F}[u]\right)\,dx=0;
\]

\item
$S_{k+1,F}[u]=\binom{n}{k+1}$
almost everywhere in $\Omega$;

\item
$A_F[u]=I$ almost everywhere in $\Omega$;

\item
there exist $x_0\in\R^n$ and $c>0$ such that
the domain and solution are given by
\eqref{eq:main-conclusion}.
\end{enumerate}
\end{theorem}

We now outline the main ideas and difficulties in the proof of Theorem~\ref{thm:main}.
For $k<n$, prescribing $S_k(A)$ does not determine $\det A$,
so the transport argument for $k=n$~\cite{b8} does not
apply directly. Moreover, boundary curvature terms on
general domains have no \emph{a priori} sign.
We combine divergence-free Newton tensors, signed curvature
integrals and the anisotropic Minkowski and Pohozaev
identities with the estimate
$S_{k+1,F}[u]\le\binom{n}{k+1}$ for the normalized equation
on $\Gamma_k$.
This yields an exact identity with three nonnegative
deficits, whose vanishing forces $A_F[u]=I$.
Polar duality then recovers the solution and the Wulff
domain, without requiring domain convexity or
$(k+1)$-admissibility.
The remaining results follow from short additional
arguments: positive conormal data select the admissible
branch through a contact-point argument, while scaling
and convexity give the two corollaries.
The equivalent characterizations follow from the same
equality case.

The paper is organized as follows.
Section~\ref{sec:preliminaries} collects the duality relations
and identities for Newton tensors and signed anisotropic
curvature integrals.
Section~\ref{sec:identities} establishes the sign properties
and Pohozaev identity and justifies integration near
critical points.
Section~\ref{sec:neumann} proves the boundary equivalences
and automatic admissibility for the conormal problem.
Section~\ref{sec:proof} presents the proofs of the main
results, deriving the deficit identity within the proof
of Theorem~\ref{thm:main}.

\section{Anisotropic Hessian preliminaries}\label{sec:preliminaries}

\subsection{Duality identities}
\noindent

We use the polar energy
\begin{equation}\label{eq:E0-definition}
E_0(x)=\frac12F_0(x)^2.
\end{equation}
We first record the regularity and duality facts needed at critical points.
For $\xi\ne0$, we set $B(\xi)=D_\xi^2E(\xi)$. Differentiating the
$2$-homogeneity of $E$ yields
\begin{equation}\label{eq:B-homogeneity}
B(t\xi)=B(\xi),\qquad B(\xi)\xi=D_\xi E(\xi)
\quad(t>0,\ \xi\ne0).
\end{equation}
Since $B$ is continuous and positive definite on the unit sphere, there
are constants $0<\lambda\le\Lambda$ such that
\begin{equation}\label{eq:B-ellipticity}
\lambda I\le B(\xi)\le\Lambda I\qquad(\xi\ne0).
\end{equation}
We recall the Euler's identity
\begin{equation}\label{eq:euler-E}
\xi\cdot D_\xi E(\xi)=2E(\xi)=F(\xi)^2.
\end{equation}
In particular, \eqref{eq:B-homogeneity}--\eqref{eq:euler-E} imply that
$\lambda|\xi|^2\le F(\xi)^2\le\Lambda|\xi|^2$ and
$|D_\xi E(\xi)|\le\Lambda|\xi|$. Thus $E$ is differentiable at $0$ with $D_\xi E(0)=0$  and this extension of its gradient is continuous.
The identity \eqref{eq:euler-E} also holds at $0$.

For $p,q\in\R^n$ with $p\ne q$, the chain rule gives
\[
\frac{\mathrm d}{\mathrm dt}
D_\xi E\bigl(q+t(p-q)\bigr)
=B\bigl(q+t(p-q)\bigr)(p-q)
\]
whenever $q+t(p-q)\ne0$.
The continuity of $D_\xi E$ and boundedness of $B$  yield
\[
D_\xi E(p)-D_\xi E(q)
=\int_0^1 B\bigl(q+t(p-q)\bigr)(p-q)\dd t,
\]
where the possible parameter with $q+t(p-q)=0$ is omitted.
Taking the inner product with $p-q$, we obtain
\[
\begin{aligned}
&\bigl(D_\xi E(p)-D_\xi E(q)\bigr)\cdot(p-q)\\
&\qquad=
\int_0^1(p-q)\cdot
B\bigl(q+t(p-q)\bigr)(p-q)\dd t.
\end{aligned}
\]
The bounds \eqref{eq:B-ellipticity} therefore imply
\begin{equation}\label{eq:strong-monotonicity}
\lambda|p-q|^2
\le
\bigl(D_\xi E(p)-D_\xi E(q)\bigr)\cdot(p-q)
\le\Lambda|p-q|^2.
\end{equation}
Moreover, the same integral representation gives
\begin{equation}\label{eq:gradient-lipschitz}
\begin{aligned}
|D_\xi E(p)-D_\xi E(q)|
&\le\int_0^1
\left|B\bigl(q+t(p-q)\bigr)(p-q)\right|\dd t\\
&\le\int_0^1\Lambda|p-q|\dd t
=\Lambda|p-q|.
\end{aligned}
\end{equation}
Consequently, $E\in C^{1,1}(\R^n)$ and $D_\xi E$ is strongly monotone.
Integrating the lower bound once more leads to 
\[
E(p)\ge E(q)+D_\xi E(q)\cdot(p-q)
+\frac{\lambda}{2}|p-q|^2,
\]
so the extension of $E$ is strongly convex on all of $\R^n$.
In particular, $D_\xi E$ is injective.

The polar inequality $x\cdot\xi\le F_0(x)F(\xi)$ and positive
homogeneity give
\[
\sup_{\xi\in\R^n}\bigl(x\cdot\xi-E(\xi)\bigr)
=\sup_{t\ge0}\left(tF_0(x)-\frac{t^2}{2}\right)
=\frac12F_0(x)^2.
\]
Indeed, the supremum of $x\cdot\eta$ on $F(\eta)=1$ is attained and equals $F_0(x)$. Thus $E_0$ is the Legendre conjugate of $E$.
The function $E(\xi)-x\cdot\xi$ is coercive and strongly convex, so it
has a unique minimizer $\xi(x)$, characterized by
$D_\xi E(\xi(x))=x$. Estimate \eqref{eq:strong-monotonicity} yields
$|\xi(x)-\xi(y)|\le\lambda^{-1}|x-y|$.
The maximizing property in the preceding display gives
\[
0\le E_0(x+h)-E_0(x)-\xi(x)\cdot h
\le\bigl(\xi(x+h)-\xi(x)\bigr)\cdot h
\le\lambda^{-1}|h|^2.
\]
It follows that $DE_0(x)=\xi(x)$ for every $x$, including $x=0$.
Hence $E_0\in C^{1,1}(\R^n)$ and
\begin{equation}\label{eq:gradient-inverse}
(D_\xi E)^{-1}=DE_0.
\end{equation}
Away from $0$, the inverse function theorem applies because $B$ is
positive definite. It gives $E_0\in C^3(\R^n\setminus\{0\})$ and
$D^2E_0(x)=B(DE_0(x))^{-1}$.

At a conjugate pair $x=D_\xi E(\xi)$, Euler's identity gives
$E_0(x)=x\cdot\xi-E(\xi)=E(\xi)$. Therefore, we obtain
\begin{align}
E_0(D_\xi E(\xi))&=E(\xi),
\label{eq:E-duality}\\
F(DE_0(x))&=F_0(x),
\label{eq:norm-gradient-duality}\\
D_\xi E(DE_0(x))&=x.
\label{eq:gradient-composition}
\end{align}
The middle equality follows from \eqref{eq:E-duality}
by setting $\xi=DE_0(x)$ and taking nonnegative square roots
and all three identities remain valid at the origin.
 In particular, the model
profile $u_0=E_0-\tfrac12$ satisfies
$D_\xi E(Du_0(x))=x$ everywhere. Differentiating this flux gives
$A_F[u_0]=I$ in the weak sense, even if $u_0$ is not twice
differentiable at $0$.

At points where $Du\ne0$,
$A_F[u]=B(Du)D^2u$ is similar to the symmetric matrix
\begin{equation}\label{eq:similar-matrix}
B(Du)^{-1/2}A_F[u]B(Du)^{1/2}
=B(Du)^{1/2}D^2uB(Du)^{1/2}.
\end{equation}
Thus its eigenvalues are real. They are nonnegative if $u$ is convex,
and positive if $D^2u$ is positive definite. The arguments below use
the weaker condition that this eigenvalue vector belongs to $\Gamma_k$.

\subsection{Newton tensors}
\noindent

For any real matrix $A=(a_{ij})$, the polynomial $S_k(A)$ can be
written as the sum of its principal minors of order $k$. Set
\begin{equation}\label{eq:newton-tensor}
S_k^{ij}(A)=\frac{\partial S_k(A)}{\partial a_{ij}},
\end{equation}
and use the convention $S_0(A)=1$. Using the Homogeneity and the Newton
recurrence, we obtain 
\begin{align}
S_k^{ij}(A)a_{ij}&=kS_k(A),
\label{eq:newton-euler}\\
S_{k+1}^{ij}(A)&=S_k(A)\delta_{ij}
-\sum_{\ell=1}^nS_k^{i\ell}(A)a_{j\ell},
\qquad 1\le k<n.
\label{eq:newton-recurrence}
\end{align}
For $A=A_F[u]$, we  set
\[
T_k^{ij}[u]=S_k^{ij}(A_F[u]).
\]
This index convention is important because $A_F[u]$ need not be symmetric.
At noncritical points, $A_F[u]B(Du)=B(Du)A_F[u]^T$; consequently,
the polynomial observation above gives
$B(Du)T_k[u]=T_k[u]^TB(Du)$.

\begin{lemma}\label{lem:div-free}
Away from the critical set of $u$,
\begin{equation}\label{eq:div-free}
\partial_jT_k^{ij}[u]=0,\qquad i=1,\ldots,n.
\end{equation}
Consequently,
\begin{equation}\label{eq:Sk-divergence}
S_{k,F}[u]=\frac1k
\partial_j\bigl(T_k^{ij}[u]D_{\xi_i}E(Du)\bigr).
\end{equation}
\end{lemma}

\begin{proof}
Write $v=D_\xi E(Du)$ locally, so that
$(A_F[u])_{ij}=\partial_jv_i$.
The assumed regularity gives $v\in C^2$ away from the critical set.
For $k=1$, $T_1^{ij}=\delta_{ij}$ and the assertion is immediate.
For $k\ge2$, the principal-minor formula gives
\[
T_k^{ij}[u]=\frac1{(k-1)!}
\delta_{i\,i_2\ldots i_k}^{j\,j_2\ldots j_k}
(\partial_{j_2}v_{i_2})\cdots(\partial_{j_k}v_{i_k}),
\]
where the generalized Kronecker symbol is alternating in its upper
indices and in its lower indices.
In each term of $\partial_jT_k^{ij}$, one factor becomes
$\partial_j\partial_{j_s}v_{i_s}$ for some $s\ge2$.
This factor is symmetric in $j,j_s$, whereas the Kronecker symbol
is antisymmetric in these two indices. Summing over them therefore
gives zero. This proves \eqref{eq:div-free}.
The product rule and \eqref{eq:newton-euler} now give
\[
\partial_j(T_k^{ij}[u]v_i)
=T_k^{ij}[u]\partial_jv_i=kS_{k,F}[u],
\]
which proves \eqref{eq:Sk-divergence}.
This is the Newton-tensor calculation for nonsymmetric gradient
matrices in \cite[Proposition~2.3]{b9}.
\end{proof}

\begin{lemma}\label{lem:newton}
Let $A=BC$, where $B$ is symmetric positive definite and $C$ is
symmetric. If $\lambda(A)\in\Gamma_m$, then, for $1\le j<\ell\le m$,
\begin{equation}\label{eq:maclaurin}
\left(\frac{S_j(A)}{\binom{n}{j}}\right)^{1/j}
\ge
\left(\frac{S_\ell(A)}{\binom{n}{\ell}}\right)^{1/\ell}.
\end{equation}
Equality at an adjacent stage holds if and only if
$A=\alpha I$ for some $\alpha>0$.
\end{lemma}

\begin{proof}
The matrices $A$ and $B^{1/2}CB^{1/2}$ have the same eigenvalues.
For a vector in $\Gamma_m$, the Newton inequalities for normalized
elementary symmetric polynomials imply that
\[
\left(\frac{S_j}{\binom{n}{j}}\right)^{1/j},
\qquad 1\le j\le m,
\]
is nonincreasing. This gives \eqref{eq:maclaurin}.
The equality statement in the adjacent Newton--Maclaurin inequality
on this open cone says that all eigenvalues coincide; see
\cite[Section~2]{b24}. Their common value is positive because
$S_1(A)>0$. Since $B^{1/2}CB^{1/2}$ is symmetric, equality makes it
$\alpha I$. Hence $C=\alpha B^{-1}$ and $A=\alpha I$.
The converse follows by direct substitution.
\end{proof}

The estimate at the next order does not require
$(k+1)$-admissibility.

\begin{lemma}\label{lem:one-step-newton}
Let $1\le k<n$, let $\lambda\in\Gamma_k$, and assume
$S_k(\lambda)=\binom{n}{k}$. Then
\begin{equation}\label{eq:one-step-bounds}
\begin{aligned}
S_1(\lambda)&\ge n,\\
S_{k-1}(\lambda)&\ge\binom{n}{k-1}\quad(k\ge2),\\
S_{k+1}(\lambda)&\le\binom{n}{k+1}.
\end{aligned}
\end{equation}
Equality in the last inequality holds if and only if
$\lambda=(1,\ldots,1)$.
\end{lemma}

\begin{proof}
The first two inequalities follow from \eqref{eq:maclaurin} and the
normalization of $S_k$. If $S_{k+1}(\lambda)\le0$, the last inequality
is strict because $\binom{n}{k+1}>0$.
Otherwise, $S_j(\lambda)>0$ for every $1\le j\le k+1$, so
$\lambda\in\Gamma_{k+1}$ by definition. The adjacent
Newton--Maclaurin inequality then gives
\[
\left(\frac{S_{k+1}(\lambda)}{\binom{n}{k+1}}\right)^{1/(k+1)}
\le
\left(\frac{S_k(\lambda)}{\binom{n}{k}}\right)^{1/k}=1.
\]
Equality forces all components of $\lambda$ to coincide.
Their common value is positive, and $S_k(\lambda)=\binom{n}{k}$
therefore makes it $1$. Conversely, $\lambda=(1,\ldots,1)$ gives
equality.
\end{proof}

\subsection{Anisotropic curvatures and signed curvature integrals}
\noindent

Let $\Sigma$ be an oriented smooth hypersurface with Euclidean
unit normal $\nu$. The anisotropic normal is $D_\xi F(\nu)$.
The restriction of $D_\xi^2F(\nu)$ to $\nu^\perp$ is
\[
\mathcal A_F(\nu)=D_{\mathbb S^{n-1}}^2F(\nu)+F(\nu)I.
\]
It is positive definite. Indeed, for a nonzero tangent vector $\tau$,
put $z=\tau-(D_\xi F(\nu)\cdot\tau)\nu/F(\nu)$.
Then $z\ne0$, $D_\xi F(\nu)\cdot z=0$, and
$D_\xi^2F(\nu)\nu=0$. Since
$B=D_\xi F\otimes D_\xi F+FD_\xi^2F$, it follows that
\[
D_\xi^2F(\nu)[\tau,\tau]
=D_\xi^2F(\nu)[z,z]
=\frac{B(\nu)[z,z]}{F(\nu)}>0.
\]
We use the sign convention that $d\nu$ is positive on an
outward-oriented Euclidean sphere, and define
\begin{equation}\label{eq:anisotropic-weingarten}
\mathcal S_F=\mathcal A_F(\nu)\circ d\nu.
\end{equation}
This map is similar to the symmetric map
$\mathcal A_F(\nu)^{1/2}d\nu\,\mathcal A_F(\nu)^{1/2}$.
Its real eigenvalues $\kappa_1^F,\ldots,\kappa_{n-1}^F$ are the
anisotropic principal curvatures. We write
\[
S_j(\kappa_F)=S_j(\kappa_1^F,\ldots,\kappa_{n-1}^F),
\qquad S_0(\kappa_F)=1.
\]
For a regular level set of $u$, we orient the normal by
$\nu=Du/|Du|$. Then
\begin{equation}\label{eq:curvature-identity}
S_{j-1}(\kappa_F)
=\frac{T_j^{i\ell}[u]u_\ell F_{\xi_i}(Du)}{F(Du)^j},
\qquad 1\le j\le n.
\end{equation}
Here this formula requires neither convexity of the level set nor
reversibility of $F$. For completeness, fix a point on the level set,
write $\xi=Du$, and choose an orthonormal frame with $\nu=e_n$.
At this point, set
\[
L=F(\xi)D_\xi^2F(\xi)D^2u.
\]
The last row of $L$ is zero, and
\[
A_F[u]=L+D_\xi F(\xi)\otimes(D^2uD_\xi F(\xi)).
\]
For a tangent vector $\tau$, differentiation of
$\nu=Du/|Du|$ and the homogeneity of $D_\xi^2F$ give
$L\tau=F(\xi)\mathcal S_F\tau$.
Thus the tangential block of $L$ has eigenvalues
$F(\xi)\kappa_1^F,\ldots,F(\xi)\kappa_{n-1}^F$.

Consider the derivative of $S_j$ in the matrix direction
$D_\xi F(\xi)\otimes\xi$.
Only the last column is changed, because $\xi=|\xi|e_n$.
In a principal minor containing that column, differentiation
replaces it by $|\xi|D_\xi F(\xi)$.
Every contribution from the rank-one part of any other column
vanishes, since it would give two proportional columns.
The remaining columns are those of $L$.
Expansion along the last row therefore contributes
$|\xi|F_{\xi_n}(\xi)=F(\xi)$ times the corresponding
tangential principal minor of order $j-1$.
Summing these minors gives
\[
T_j^{i\ell}[u]F_{\xi_i}(\xi)\xi_\ell
=F(\xi)S_{j-1}(L|_{T\Sigma})
=F(\xi)^jS_{j-1}(\kappa_F),
\]
which is \eqref{eq:curvature-identity}; compare
\cite[Theorem~2.5]{b9}. The calculation uses only positive
homogeneity and the chosen orientation.

For a bounded $C^2$ domain, orient every component of
$\partial\Omega$ by the outward normal of $\Omega$ and define
\begin{equation}\label{eq:signed-curvature-functional}
\begin{aligned}
\mathcal W_{0,F}(\Omega)&=|\Omega|,\\
\mathcal W_{j,F}(\Omega)
&=\frac1{n\binom{n-1}{j-1}}
\int_{\partial\Omega}S_{j-1}(\kappa_F)F(\nu)
\dd\mathcal H^{n-1},\qquad 1\le j\le n.
\end{aligned}
\end{equation}
For convex domains these are the corresponding mixed volumes.
For general domains they are signed curvature integrals; their
positivity is not part of the definition.

\begin{lemma}\label{lem:signed-minkowski}
Let $\Sigma$ be a closed oriented $C^3$ hypersurface.
For $1\le r\le n-1$,
\begin{equation}\label{eq:signed-minkowski-r}
r\int_\Sigma(x\cdot\nu)S_r(\kappa_F)\dd\mathcal H^{n-1}
=(n-r)\int_\Sigma F(\nu)S_{r-1}(\kappa_F)
\dd\mathcal H^{n-1}.
\end{equation}
The identity holds on each component without convexity or
star-shapedness.
\end{lemma}

\begin{proof}
We translate the normal convention in the anisotropic
Hsiung--Minkowski formula of \cite[Theorem~1.1]{b14}.
Use the reflected anisotropy $\bar F(\eta)=F(-\eta)$ and the
opposite normal $-\nu$. The tangent spaces at $\nu$ and $-\nu$
are the same subspace of $\R^n$, and
\[
\bar F(-\nu)=F(\nu),\qquad
D_{\mathbb S^{n-1}}^2\bar F(-\nu)
+\bar F(-\nu)I=\mathcal A_F(\nu).
\]
Consequently, the Weingarten map in the convention of that formula is
\[
-\bigl(D_{\mathbb S^{n-1}}^2\bar F(-\nu)
+\bar F(-\nu)I\bigr)\circ d(-\nu)=\mathcal S_F.
\]
The cited formula with index $r-1$ therefore reads
\[
0=\int_\Sigma\left[
F(\nu)\frac{S_{r-1}(\kappa_F)}{\binom{n-1}{r-1}}
-(x\cdot\nu)\frac{S_r(\kappa_F)}{\binom{n-1}{r}}
\right]\dd\mathcal H^{n-1}.
\]
Since
$\binom{n-1}{r}/\binom{n-1}{r-1}=(n-r)/r$,
this is exactly \eqref{eq:signed-minkowski-r}.
The argument does not require $F$ to be even, and the identity
applies separately to every oriented component.
If the cited formula is stated for smooth data, approximate $F$ on the
unit sphere and the hypersurface in the $C^2$ topology. Positivity of
$\mathcal A_F$ is preserved for sufficiently close approximations.
The integrands involve only $F$, its first two derivatives and the
first derivative of the normal, so the formula passes to the present
$C^3$ data by uniform convergence.
\end{proof}

The divergence theorem for $j=1$, and
\eqref{eq:signed-minkowski-r} with $r=j-1$ for $j\ge2$, give
\begin{equation}\label{eq:minkowski-identity}
\int_{\partial\Omega}(x\cdot\nu)S_{j-1}(\kappa_F)
\dd\mathcal H^{n-1}
=n\binom{n-1}{j-1}\mathcal W_{j-1,F}(\Omega),
\qquad1\le j\le n.
\end{equation}
For $j\ge2$, the coefficient follows from
$(n-j+1)\binom{n-1}{j-2}/(j-1)=\binom{n-1}{j-1}$.
All boundary components retain the outward orientation, including
components bounding holes in the domain.

\section{Integral identities}\label{sec:identities}
\noindent

For $1\le k<n$, we derive the following identities on
a bounded $C^3$ domain $\Omega$.
We assume that $u$ satisfies \eqref{eq:regularity-class}
with a finite critical set and write
\[
v=D_\xi E(Du).
\]

By \eqref{eq:gradient-lipschitz}, $v$ is Lipschitz on $\overline\Omega$.
Its classical divergence away from $\mathcal C_u$ therefore agrees almost
everywhere with its weak divergence on $\Omega$.

\begin{lemma}\label{lem:automatic-sign}
Assume that $u$ satisfies \eqref{eq:regularity-class} and is a $k$-admissible solution of 
\[
\begin{cases}
S_{k,F}[u]=\binom{n}{k} & \text{a.e. in }\Omega,\\
u=0 & \text{on }\partial\Omega.
\end{cases}
\]
Then $u<0$ in $\Omega$. If in addition $F(Du)=c>0$ on $\partial\Omega$, then
\begin{equation}\label{eq:positive-normal-derivative}
Du=u_\nu\nu,\qquad u_\nu>0,\qquad F(Du)=u_\nu F(\nu)=c\quad\text{on }\partial\Omega.
\end{equation}
\end{lemma}

\begin{proof}
By Lemma~\ref{lem:one-step-newton}, we see that 
\[
\operatorname{div}\bigl(D_\xi E(Du)\bigr)=S_{1,F}[u]\ge n
\]
almost everywhere and hence in distributions. Since $u=0$ on the boundary,
the positive part $u_+=\max\{u,0\}$ belongs to $H_0^1(\Omega)$.
Approximation by nonnegative smooth functions allows us to use it as a test
function. Euler's identity then gives
\[
n\int_\Omega u_+\dd x\le-\int_\Omega D_\xi E(Du)\cdot Du_+\dd x
=-\int_{\{u>0\}}F(Du)^2\dd x\le0.
\]
Thus $u_+=0$ almost everywhere, and continuity gives $u\le0$ everywhere.

To obtain strict negativity without a classical Hessian at critical points,
fix $x_*\in\Omega$ and choose $r>0$ so small that
$x_*+r\overline{\Wulff}\subset\Omega$. On this Wulff ball set
\[
b(x)=E_0(x-x_*)-\frac{r^2}{2}.
\]
Then $b=0$ on the boundary of this Wulff ball and, by \eqref{eq:gradient-composition},
\[
D_\xi E(Db)=x-x_*,
\qquad\operatorname{div}\bigl(D_\xi E(Db)\bigr)=n.
\]
These identities hold in the weak sense also at $x_*$, because the flux
is exactly the affine field $x-x_*$.
Since $u\le b$ on the boundary of the Wulff ball, $(u-b)_+$ belongs to its $H_0^1$ space. Testing
the distributional inequality $\operatorname{div}(D_\xi E(Du))\ge\operatorname{div}(D_\xi E(Db))$ with $(u-b)_+$ and applying
\eqref{eq:strong-monotonicity} give
\[
0\ge\int_{\{u>b\}}\bigl(D_\xi E(Du)-D_\xi E(Db)\bigr)\cdot(Du-Db)\dd x
\ge\lambda\int_{\{u>b\}}|Du-Db|^2\dd x\ge0.
\]
It follows that $D(u-b)_+=0$ almost everywhere. Its zero trace and
Poincar\'e's inequality give $(u-b)_+=0$, hence $u\le b$ throughout the ball.
In particular, $u(x_*)\le b(x_*)=-r^2/2<0$. Since $x_*$ was arbitrary,
$u<0$ in $\Omega$.

On $\partial\Omega$, the zero Dirichlet condition gives $Du=u_\nu\nu$. At each boundary point $p$, the
interior difference quotient $[u(p)-u(p-h\nu(p))]/h$ is nonnegative for sufficiently small
$h>0$. Passing to the limit gives $u_\nu(p)\ge0$. The condition $F(Du)=c>0$ excludes
$u_\nu=0$, so $u_\nu>0$. Positive homogeneity gives \eqref{eq:positive-normal-derivative}; reversibility is not needed.
\end{proof}

\begin{lemma}\label{lem:boundary-Sj}
Assume the hypotheses of Lemma~\ref{lem:automatic-sign} and, in addition, $F(Du)=c>0$
on $\partial\Omega$. Then, for $1\le j\le n$, we have 
\begin{equation}\label{eq:boundary-Sj}
\int_\Omega S_{j,F}[u]\dd x=c^j\binom nj\mathcal W_{j,F}(\Omega).
\end{equation}
\end{lemma}

\begin{proof}
Since $u=0$ on $\partial\Omega$ and $u<0$ in $\Omega$, we have $Du=u_\nu\nu$ on the boundary. The
condition $F(Du)=c>0$ gives $u_\nu>0$. Hence
\begin{equation}\label{eq:normal-derivative}
u_\nu F(\nu)=F(Du)=c.
\end{equation}
By \eqref{eq:Sk-divergence}, \eqref{eq:curvature-identity} and \eqref{eq:normal-derivative}, we know that 
\begin{align*}  
\int_\Omega S_{j,F}[u]\dd x
&=\frac1j\int_{\partial\Omega}T_j^{i\ell}v_i\nu_\ell\dd\mathcal H^{n-1}\\
&=\frac{c^j}{j}\int_{\partial\Omega}S_{j-1}(\kappa_F)F(\nu)\dd\mathcal H^{n-1}.
\end{align*}
For clarity, the boundary contraction in the second equality is
\[
T_j^{i\ell}[u]v_i\nu_\ell
=\frac{F(Du)^{j+1}}{u_\nu}S_{j-1}(\kappa_F)
=c^jF(\nu)S_{j-1}(\kappa_F).
\]
The divergence theorem is first applied after deleting small balls around
$\mathcal C_u$. On a sphere of radius $\varepsilon$ about a critical point,
$v=O(\varepsilon)$ and $T_j[u]=O(1)$, so its flux is
$O(\varepsilon^n)$ and vanishes in the limit. The volume integrands are
bounded, so the omitted volumes contribute nothing. The estimates are
recorded in Remark~\ref{rem:critical-set}.
Finally, \eqref{eq:signed-curvature-functional} and
$\frac nj\binom{n-1}{j-1}=\binom{n}{j}$ give \eqref{eq:boundary-Sj}.
\end{proof}

\begin{remark}\label{rem:critical-set}
We record the estimates used when integration passes across the critical set.
For sufficiently small $\varepsilon>0$, remove disjoint closed Euclidean
balls of radius $\varepsilon$ about all points of $\mathcal C_u$.
These balls are compactly contained in $\Omega$.
If $a\in\mathcal C_u$, then $Du(a)=0$ and the Lipschitz continuity of $Du$
gives $|Du(x)|=O(|x-a|)$. By \eqref{eq:gradient-lipschitz} and
$D_\xi E(0)=0$, this also gives $|v(x)|=O(|x-a|)$.
The bounds \eqref{eq:B-ellipticity} and $u\in C^{1,1}$ imply
\[
|D_\xi^2E(Du)|+|D^2u|+|A_F[u]|+|T_j[u]|=O(1)
\]
on the punctured domain. Although the Hessian bound is initially an
almost-everywhere bound, continuity of $D^2u$ away from $\mathcal C_u$
makes it valid there pointwise.

Each artificial boundary has area $O(\varepsilon^{n-1})$.
A flux containing $T_j[u]v$ is therefore $O(\varepsilon^n)$.
If $\varphi$ and $D\varphi$ are bounded, a flux containing
$uT_j[u]D_\xi^2E(Du)D\varphi$ is $O(\varepsilon^{n-1})$.
The same bound applies to $(x\cdot\nu)(-u)S_{j,F}[u]$.
Since there are only finitely many critical points and $n\ge2$, all these
contributions vanish as $\varepsilon\downarrow0$.
No bound on $D^3u$ near the critical points is needed.

In particular, the divergence identity for $T_j[u]v$ extends across
$\mathcal C_u$ in distributions: its flux is bounded, its divergence
equals the bounded function $jS_{j,F}[u]$ off $\mathcal C_u$, and the
artificial fluxes vanish. Whenever $u=0$ on $\partial\Omega$, we may test
this identity with $u\in H_0^1(\Omega)$ and obtain
\begin{equation}\label{eq:energy-Q}
\int_\Omega(-u)S_{j,F}[u]\dd x
=\frac1j\int_\Omega T_j^{i\ell}[u]v_i u_\ell\dd x,
\qquad1\le j\le n.
\end{equation}
Approximation in $H_0^1$ justifies the test because both the flux and its
divergence are bounded. This formula uses only the divergence structure
and the zero Dirichlet data, independently of the equation for $u$.
\end{remark}

The next identity does not require admissibility. We include the outward
boundary orientation explicitly because the gauge may be non-reversible.

\begin{lemma}\label{lem:pohozaev}
Let $f\in C^1(\R)$ and define $G(s)=\int_s^0f(\tau)\dd\tau$. If
\[
\begin{cases}
S_{k,F}[u]=f(u) & \text{a.e. in }\Omega,\\
u=0,\quad u_\nu>0 & \text{on }\partial\Omega,
\end{cases}
\]
then
\begin{equation}\label{eq:pohozaev}
\begin{aligned}
&\frac{n-2k}{k(k+1)}\int_\Omega T_k^{ij}[u]v_i u_j\dd x\\
&\quad+\frac1{k+1}\int_{\partial\Omega}(x\cdot\nu)F(Du)^{k+1}S_{k-1}(\kappa_F)
\dd\mathcal H^{n-1}=n\int_\Omega G(u)\dd x.
\end{aligned}
\end{equation}
\end{lemma}

\begin{proof}
We set $B=D_\xi^2E(Du)$, $S=S_{k,F}[u]$ and
$T^{ij}=T_k^{ij}[u]$, and first work on
\[
\Omega_\varepsilon
=\Omega\setminus\bigcup_{a\in\mathcal C_u}
\overline{B_\varepsilon(a)},
\]
where $\varepsilon>0$ is small enough that the removed
closed balls are pairwise disjoint and compactly
contained in $\Omega$.
The nonvanishing boundary gradient
and compactness give a positive lower bound for $|Du|$ on
$\overline{\Omega_\varepsilon}$. Here $v$ is $C^2$ and $A_F[u]=Dv$ is $C^1$,
and continuity makes the almost-everywhere equation pointwise. Define
\begin{equation}\label{eq:Ek-functional}
\mathscr E_{k,\varepsilon}(w)=\frac1{k+1}\int_{\Omega_\varepsilon}(-w)S_{k,F}[w]\dd x.
\end{equation}
For a $C^2$ variation $\varphi$ on this punctured domain, $D(u+t\varphi)$
remains nonzero for sufficiently small $|t|$. We may therefore
differentiate the integral at $t=0$. Differentiating
$A_{ij}=\partial_jv_i$ gives
\[
\dot A_{ij}=\partial_j(B_{i\ell}\varphi_\ell),\qquad
\dot S=T^{ij}\partial_j(B_{i\ell}\varphi_\ell).
\]
The matrix $A_F[u]=BD^2u$ is self-adjoint for the inner product induced by $B^{-1}$, or
equivalently, $BA_F[u]^T=A_F[u]B$. Since $T$ is a polynomial in $A_F[u]^T$,
\begin{equation}\label{eq:B-self-adjoint-newton}
BT=T^TB.
\end{equation}
Euler's identity $B\,Du=v$ implies
\[
u_jT^{ij}B_{i\ell}=T^{i\ell}v_i,\qquad\partial_\ell(T^{i\ell}v_i)=kS.
\]
Using $\partial_jT^{ij}=0$ for the first integration by parts gives
\[
\begin{aligned}
(k+1)\delta\mathscr E_{k,\varepsilon}(u)[\varphi]
={}&-\int_{\Omega_\varepsilon}\varphi S\dd x
+\int_{\Omega_\varepsilon}T^{i\ell}v_i\varphi_\ell\dd x\\
&-\int_{\partial\Omega_\varepsilon}
uT^{ij}B_{i\ell}\varphi_\ell\nu_j\dd\mathcal H^{n-1}.
\end{aligned}
\]
A second integration by parts, using
$\partial_\ell(T^{i\ell}v_i)=kS$, then yields
\begin{align}
\delta\mathscr E_{k,\varepsilon}(u)[\varphi]
={}&-\int_{\Omega_\varepsilon}\varphi S\dd x\notag\\
&+\frac1{k+1}\int_{\partial\Omega_\varepsilon}
\bigl(\varphi T^{i\ell}v_i\nu_\ell-uT^{ij}B_{i\ell}\varphi_\ell\nu_j\bigr)
\dd\mathcal H^{n-1}.
\label{eq:first-variation-with-boundary}
\end{align}
For a variation whose value and first derivatives are bounded near
$\mathcal C_u$, let $\varepsilon\downarrow0$. On each artificial boundary,
\[
|v|=O(\varepsilon),\qquad|B|+|T|+|u|+|D\varphi|=O(1).
\]
The artificial-boundary terms in \eqref{eq:first-variation-with-boundary} are respectively $O(\varepsilon^n)$ and $O(\varepsilon^{n-1})$, so both vanish
since $n\ge2$. Since $u=0$ on $\partial\Omega$, we obtain
\begin{equation}\label{eq:first-variation-boundary}
\lim_{\varepsilon\downarrow0}\delta\mathscr E_{k,\varepsilon}(u)[\varphi]
=-\int_\Omega\varphi S\dd x+\frac1{k+1}\int_{\partial\Omega}
\varphi T^{i\ell}v_i\nu_\ell\dd\mathcal H^{n-1}.
\end{equation}

We next evaluate the same variation with $\varphi=x\cdot Du$.
It is $C^2$ on each punctured domain, and both $\varphi$ and
$D\varphi=Du+D^2u\,x$ are bounded uniformly near $\mathcal C_u$.
Euler's identity gives
\[
B_{i\ell}\varphi_\ell=v_i+x_m\partial_mv_i
\]
and equality of mixed derivatives of $v_i$ gives
\[
\dot A_{ij}=2A_{ij}+x_m\partial_mA_{ij},\qquad\dot S=2kS+x\cdot DS.
\]
Only first derivatives of $A_F[u]$ occur in these formulas.
Thus the assumed $C^3$ regularity of $F$ away from the origin and of $u$
away from $\mathcal C_u$ is sufficient; no fourth derivatives are used.
Since $S=f(u)$, the functions $S$ and $DS=f'(u)Du$ are bounded.
Differentiating \eqref{eq:Ek-functional} along this variation therefore gives
\begin{align}
\lim_{\varepsilon\downarrow0}\delta\mathscr E_{k,\varepsilon}(u)[x\cdot Du]
&=\frac1{k+1}\int_\Omega\left[x\cdot D((-u)S)+2k(-u)S\right]\dd x\notag\\
&=\frac{2k-n}{k+1}\int_\Omega(-u)S\dd x.
\label{eq:dilation-energy-variation}
\end{align}
The divergence theorem is first applied on $\Omega_\varepsilon$.
The outer boundary term is zero because $u=0$; the artificial boundary
terms are $O(\varepsilon^{n-1})$ and vanish. Boundedness of the volume
integrands justifies passage to the limit. Combining
\eqref{eq:first-variation-boundary} and \eqref{eq:dilation-energy-variation} gives
\begin{equation}\label{eq:pohozaev-preliminary}
\begin{aligned}
\int_\Omega(x\cdot Du)S\dd x
={}&\frac{n-2k}{k+1}\int_\Omega(-u)S\dd x\\
&+\frac1{k+1}\int_{\partial\Omega}
(x\cdot\nu)u_\nu T^{i\ell}v_i\nu_\ell\dd\mathcal H^{n-1}.
\end{aligned}
\end{equation}
On $\partial\Omega$, \eqref{eq:curvature-identity} and $F(Du)=u_\nu F(\nu)$ give
\begin{equation}\label{eq:boundary-pohozaev-density}
u_\nu T^{i\ell}v_i\nu_\ell=F(Du)^{k+1}S_{k-1}(\kappa_F).
\end{equation}
Since $S=f(u)$ and $G'(u)=-f(u)$,
\begin{equation}\label{eq:potential-pohozaev}
\int_\Omega(x\cdot Du)S\dd x=-\int_\Omega x\cdot DG(u)\dd x=n\int_\Omega G(u)\dd x,
\end{equation}
where the boundary integral of $G(u)$ is zero because $G(0)=0$.
Taking $j=k$ in \eqref{eq:energy-Q}, and substituting that identity,
\eqref{eq:boundary-pohozaev-density} and \eqref{eq:potential-pohozaev}
into \eqref{eq:pohozaev-preliminary}, proves \eqref{eq:pohozaev}.
\end{proof}

\begin{remark}
When $k=1$, formula \eqref{eq:pohozaev} is the usual Pohozaev identity for the Finsler
Laplacian. When $F(\xi)=|\xi|$, it reduces to the Hessian Pohozaev identity used in \cite{b5}.
\end{remark}

\section{Automatic admissibility}\label{sec:neumann}
\noindent

We first verify the boundary equivalences \eqref{eq:N-equiv}.
The zero Dirichlet condition gives $Du=u_\nu\nu$ on
$\partial\Omega$. Since $D_\xi E=F D_\xi F$ away from the
origin, with $D_\xi E(0)=0$, Euler's identity yields
\[
u_\nu D_\xi E(Du)\cdot\nu
=Du\cdot D_\xi E(Du)=F(Du)^2.
\]
Under \eqref{eq:N-flux}, the flux is $cF(\nu)>0$.
The identity $D_\xi E(0)=0$ first excludes $Du=0$ on the
boundary. The preceding equality therefore implies
$u_\nu>0$. Positive homogeneity and Euler's identity for
$F$ now give
\[
F(Du)D_\xi F(Du)\cdot\nu
=u_\nu F(\nu)D_\xi F(\nu)\cdot\nu
=u_\nu F(\nu)^2.
\]
Consequently, \eqref{eq:N-flux} is equivalent to
$u_\nu=c/F(\nu)>0$, and hence to $F(Du)=c$ together with
$u_\nu>0$. Conversely, either of these latter conditions
gives the same flux formula. This proves all implications
in \eqref{eq:N-equiv}, using positive homogeneity only.

The next lemma supplies the admissibility required by
Theorem~\ref{thm:main}. Its contact-point argument follows
the Euclidean argument in \cite[Section~4]{b5}; the symmetric
representative \eqref{eq:similar-matrix} permits its use for
the anisotropic operator.

\begin{lemma}\label{lem:N-adm}
Let $1\le k<n$, let $\Omega$ be as in
Theorem~\ref{thm:neumann}, and suppose that $u$ satisfies
\eqref{eq:regularity-class} and has finitely many critical
points in $\Omega$. Assume that, for some $\mu>0$,
\begin{equation}\label{eq:N-hyp}
\begin{cases}
S_{k,F}[u]=\mu & \text{a.e. in }\Omega,\\
u=0 & \text{on }\partial\Omega,\\
u_\nu>0 & \text{on }\partial\Omega.
\end{cases}
\end{equation}
Then
\begin{equation}\label{eq:N-adm}
\lambda(A_F[u](x))\in\Gamma_k
\qquad\text{for every }x\in\Omega\setminus\mathcal C_u.
\end{equation}
In particular, $u$ is $k$-admissible almost everywhere
in $\Omega$.
\end{lemma}

\begin{proof}
Choose $z\in\Omega$, let
$R=\max_{\overline\Omega}|x-z|$, and choose a point $p$
where this maximum is attained. A maximum of the distance
from $z$ cannot occur in the interior, so $p\in\partial\Omega$.
The containing ball $\overline B_R(z)$ is tangent to
$\partial\Omega$ at $p$, and
$p-z=R\nu(p)$.

\begin{figure}[H]
\centering
\vspace*{-10bp}
\hspace*{24bp}\includegraphics[width=244bp]{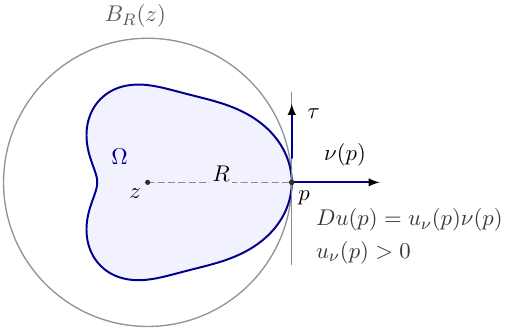}\par
\vspace{4.5bp}
\caption{Two-dimensional contact geometry used in
Lemma~\ref{lem:N-adm}. The enclosing Euclidean ball is tangent
to $\partial\Omega$ at $p$, where its outward normal agrees
with $\nu(p)$. Positive $u_\nu(p)$ makes the tangential
Hessian of $u$ positive definite at this point, without
convexity of $\Omega$.}
\label{fig:contact}
\end{figure}

To verify the curvature information at this contact point,
fix $\tau\in T_p\partial\Omega$ and a boundary curve
$\gamma$ with $\gamma(0)=p$ and $\gamma'(0)=\tau$.
Differentiating $\nu(\gamma)\cdot\gamma'=0$ gives
$\nu(p)\cdot\gamma''(0)=-D_\tau\nu\cdot\tau$.
Since $|\gamma-z|^2/2$ has a maximum at $0$,
\[
0\ge|\tau|^2+(p-z)\cdot\gamma''(0)
=|\tau|^2-RD_\tau\nu\cdot\tau.
\]
Thus
\begin{equation}\label{eq:N-contact}
D_\tau\nu\cdot\tau\ge\frac{|\tau|^2}{R}.
\end{equation}
Differentiating $u(\gamma)=0$ twice and using
$Du(p)=u_\nu(p)\nu(p)$ then yields
\begin{equation}\label{eq:N-tangent}
D^2u(p)[\tau,\tau]
=u_\nu(p)D_\tau\nu\cdot\tau
\ge\frac{u_\nu(p)}{R}|\tau|^2>0
\qquad(\tau\ne0).
\end{equation}
In particular, $D^2u(p)$ has at least $n-1$ positive
eigenvalues. Here $Du(p)\ne0$, so all derivatives used
above are classical in a neighborhood of $p$.

Write
\begin{equation}\label{eq:N-matrix}
M=B(Du)^{1/2}D^2uB(Du)^{1/2}.
\end{equation}
This matrix is symmetric and similar to $A_F[u]$.
Sylvester's law of inertia shows that $M(p)$ also has
at least $n-1$ positive eigenvalues. Both $M$ and
$S_k(M)$ are continuous wherever $Du\ne0$.
The almost-everywhere equation therefore holds at every
interior noncritical point. Taking an interior limit
at $p$ gives $S_k(M(p))=\mu$.

Denote the eigenvalues of $M(p)$ by
$a_1,\ldots,a_{n-1}>0$ and $t$, and write
$e_j=S_j(a_1,\ldots,a_{n-1})$, with $e_0=1$.
Because $k<n$, both $e_k$ and $e_{k-1}$ are positive, and
\[
0<\mu=e_k+te_{k-1}
\quad\Longrightarrow\quad
t>-\frac{e_k}{e_{k-1}}.
\]
The Newton inequalities for the positive numbers
$a_1,\ldots,a_{n-1}$ imply
$e_j/e_{j-1}\ge e_k/e_{k-1}$ for $1\le j\le k$.
It follows that
\begin{equation}\label{eq:N-seed}
S_j(M(p))
=e_j+te_{j-1}
=e_{j-1}\left(\frac{e_j}{e_{j-1}}+t\right)>0,
\qquad 1\le j\le k.
\end{equation}
Thus $\lambda(M(p))\in\Gamma_k$. Openness and continuity
give an interior point near $p$ with the same property.

We propagate this membership through
$\Omega\setminus\mathcal C_u$. This set is path connected:
an open connected subset of $\R^n$ is polygonally connected,
and, since $n\ge2$, a path can be detoured inside small
balls around its finitely many critical points.
Along any such path, $M$ is continuous and $S_k(M)=\mu$.
As long as its spectrum lies in $\Gamma_k$, the
Newton--Maclaurin inequalities give the fixed lower bounds
\begin{equation}\label{eq:N-bounds}
S_j(M)\ge
\binom{n}{j}
\left(\frac{\mu}{\binom{n}{k}}\right)^{j/k}>0,
\qquad 1\le j\le k.
\end{equation}
At a putative first exit from $\Gamma_k$, continuity
preserves all these positive lower bounds. The definition
$\Gamma_k=\{S_1>0,\ldots,S_k>0\}$ then places the limiting
spectrum in the open cone, a contradiction. Hence
\eqref{eq:N-adm} holds on the entire noncritical set.
The omitted critical set is finite and therefore has
measure zero.
\end{proof}

\begin{remark}\label{rem:N-sign}
The admissibility in Lemma~\ref{lem:N-adm} concerns the
eigenvalues of $A_F[u]$, not those of $D^2u$ alone.
The sign of the conormal flux is essential. For example,
in the Euclidean case and for every even $k$ with
$2\le k<n$, the function
\[
u(x)=\frac{c^2-|x|^2}{2}
\qquad\text{in }B_c(0)=\{x\in\R^n:|x|<c\}
\]
satisfies
\[
S_k(D^2u)=\binom{n}{k},\qquad
u=0,\quad |Du|=c\quad\text{on }\partial B_c(0).
\]
Nevertheless, $S_1(D^2u)=-n<0$ and $u_\nu=-c$.
This smooth function satisfies the stated regularity
assumptions but belongs to another branch.
Theorem~\ref{thm:neumann} uses positive conormal flux to
select the admissible branch; Theorem~\ref{thm:main}
assumes admissibility and obtains the positive boundary
orientation from the equation and boundary data.
\end{remark}
\section{Proofs of the main results}\label{sec:proof}
\noindent

We prove the results in the order stated in the introduction.
The proof of Theorem~\ref{thm:main} derives the deficit
identity directly from the hypotheses, before using an
equality case to identify the domain.

\subsection{Wulff rigidity}

\begin{proof}[Proof of Theorem~\ref{thm:main}]
We first normalize the boundary constant. Define
\[
\widetilde\Omega=c^{-1}\Omega,\qquad
\widetilde u(y)=c^{-2}u(cy).
\]
Then $D\widetilde u(y)=c^{-1}Du(cy)$.
Since $D_\xi E$ is positively $1$-homogeneous,
\[
D_\xi E(D\widetilde u(y))=c^{-1}D_\xi E(Du(cy)),
\qquad
A_F[\widetilde u](y)=A_F[u](cy).
\]
Thus the equation and admissibility are preserved, while
$F(D\widetilde u)=1$ on $\partial\widetilde\Omega$.
The regularity assumptions and finiteness of the critical
set are preserved by this scaling. We work with the
normalized pair and omit the tildes until the last step.

Lemma~\ref{lem:automatic-sign} gives $u<0$ in $\Omega$
and $u_\nu>0$ on $\partial\Omega$. In particular,
\begin{equation}\label{eq:normalized-problem}
\begin{cases}
S_{k,F}[u]=\binom{n}{k} & \text{a.e. in }\Omega,\\
u=0,\quad F(Du)=1 & \text{on }\partial\Omega.
\end{cases}
\end{equation}
In the integrations below, write $v=D_\xi E(Du)$.
All uses of integration by parts are justified by the
critical-point argument in Remark~\ref{rem:critical-set}.

We first obtain the lower bound for the signed curvature
integral. If $k\ge2$, Lemma~\ref{lem:one-step-newton} gives
\begin{equation}\label{eq:Skminus-lower}
S_{k-1,F}[u]\ge\binom{n}{k-1}
\qquad\text{a.e. in }\Omega.
\end{equation}
Integrating and applying Lemma~\ref{lem:boundary-Sj}
with $j=k-1$ and $c=1$ yields
\[
\binom{n}{k-1}\mathcal W_{k-1,F}(\Omega)
=\int_\Omega S_{k-1,F}[u]\dd x
\ge\binom{n}{k-1}|\Omega|.
\]
Therefore
\begin{equation}\label{eq:volume-curvature}
\mathcal W_{k-1,F}(\Omega)\ge|\Omega|.
\end{equation}
For $k=1$, \eqref{eq:volume-curvature} holds as an equality
by the definition
$\mathcal W_{0,F}(\Omega)=|\Omega|$.

Next, apply the Pohozaev identity with
$f\equiv\binom{n}{k}$ and
$G(u)=-\binom{n}{k}u$. The energy identity
\eqref{eq:energy-Q} gives
\[
\int_\Omega T_k^{ij}[u]v_i u_j\dd x
=k\binom{n}{k}\int_\Omega(-u)\dd x.
\]
Since $F(Du)=1$ on $\partial\Omega$, the boundary integral
in \eqref{eq:pohozaev} is determined by
\eqref{eq:minkowski-identity}:
\begin{align*}
\int_{\partial\Omega}(x\cdot\nu)S_{k-1}(\kappa_F)
\dd\mathcal H^{n-1}
&=n\binom{n-1}{k-1}\mathcal W_{k-1,F}(\Omega)\\
&=k\binom{n}{k}\mathcal W_{k-1,F}(\Omega).
\end{align*}
Substituting these two equalities into
\eqref{eq:pohozaev} gives
\[
\begin{aligned}
&\frac{n-2k}{k+1}\binom{n}{k}\int_\Omega(-u)\dd x
+\frac{k}{k+1}\binom{n}{k}\mathcal W_{k-1,F}(\Omega)\\
&\qquad=n\binom{n}{k}\int_\Omega(-u)\dd x.
\end{aligned}
\]
Cancelling the positive coefficient and collecting the
integrals of $-u$ yields
\begin{equation}\label{eq:mean-u-exact}
\int_\Omega(-u)\dd x
=\frac{\mathcal W_{k-1,F}(\Omega)}{n+2}.
\end{equation}

The trace bound in Lemma~\ref{lem:one-step-newton} reads
\begin{equation}\label{eq:trace-lower}
S_{1,F}[u]=\operatorname{div}v\ge n.
\end{equation}
Since $u=0$ on $\partial\Omega$, integration by parts
and Euler's identity imply
\begin{equation}\label{eq:dirichlet-lower}
\begin{aligned}
\int_\Omega F(Du)^2\dd x
&=\int_\Omega Du\cdot v\dd x
=\int_\Omega(-u)S_{1,F}[u]\dd x\\
&\ge n\int_\Omega(-u)\dd x
=\frac{n}{n+2}\mathcal W_{k-1,F}(\Omega).
\end{aligned}
\end{equation}
The inequality uses $-u>0$ throughout the interior.

We now derive the identity that forces these bounds to
be equalities. For this calculation only, put
$q=F(Du)^2$. This function is Lipschitz on $\overline\Omega$;
away from the critical set,
$q_\ell=2v_m u_{m\ell}$.
The Newton recurrence \eqref{eq:newton-recurrence}
and $B(Du)Du=v$ give
\begin{equation}\label{eq:Q-recurrence}
T_{k+1}^{ij}[u]v_i u_j
=S_{k,F}[u]q-\frac12T_k^{i\ell}[u]v_iq_\ell.
\end{equation}
Indeed, the term subtracted by the recurrence is
$T_k^{i\ell}v_i(A_F[u])_{j\ell}u_j$, and
$(A_F[u])_{j\ell}u_j=v_m u_{m\ell}=q_\ell/2$.
Moreover, divergence-freeness and Newton's Euler identity give
\[
\partial_\ell(T_k^{i\ell}v_i)
=kS_{k,F}[u]=k\binom{n}{k}.
\]
Using $q=1$ on the boundary, we obtain
\begin{align}
\int_\Omega T_k^{i\ell}v_iq_\ell\dd x
&=\int_{\partial\Omega}qT_k^{i\ell}v_i\nu_\ell
\dd\mathcal H^{n-1}
-k\binom{n}{k}\int_\Omega q\dd x\notag\\
&=k\binom{n}{k}|\Omega|
-k\binom{n}{k}\int_\Omega q\dd x.
\label{eq:q-integration}
\end{align}
For the last equality, apply the divergence theorem to
$T_k^{i\ell}v_i$ and use
$S_{k,F}[u]=\binom{n}{k}$.
Integrating \eqref{eq:Q-recurrence} therefore gives
\begin{equation}\label{eq:Q-exact}
\begin{aligned}
\int_\Omega T_{k+1}^{ij}[u]v_i u_j\dd x
&=\frac{k+2}{2}\binom{n}{k}\int_\Omega F(Du)^2\dd x\\
&\quad-\frac{k}{2}\binom{n}{k}|\Omega|.
\end{aligned}
\end{equation}
On the other hand, \eqref{eq:energy-Q} at order $k+1$
gives
\begin{equation}\label{eq:Q-energy-kplus}
\frac1{k+1}\int_\Omega T_{k+1}^{ij}[u]v_i u_j\dd x
=\int_\Omega(-u)S_{k+1,F}[u]\dd x.
\end{equation}
No $(k+1)$-admissibility is needed for this identity:
it follows from the divergence structure.

Combining \eqref{eq:mean-u-exact},
\eqref{eq:Q-exact}, and \eqref{eq:Q-energy-kplus},
and using
$(k+1)\binom{n}{k+1}=(n-k)\binom{n}{k}$, yields
\begin{align*}
&(k+1)\int_\Omega(-u)
\left(\binom{n}{k+1}-S_{k+1,F}[u]\right)\dd x\\
&=\frac{n-k}{n+2}\binom{n}{k}
\mathcal W_{k-1,F}(\Omega)\\
&\quad-\frac{k+2}{2}\binom{n}{k}\int_\Omega F(Du)^2\dd x
+\frac{k}{2}\binom{n}{k}|\Omega|.
\end{align*}
Moving the terms to one side and collecting the
coefficient of $\mathcal W_{k-1,F}(\Omega)$ gives
precisely \eqref{eq:deficit-identity-intro}.
This derivation uses only the hypotheses and the preceding
integral identities.

Each of its three terms is nonnegative.
Lemma~\ref{lem:one-step-newton} gives
\begin{equation}\label{eq:Skplus-upper}
S_{k+1,F}[u]\le\binom{n}{k+1}
\qquad\text{a.e. in }\Omega.
\end{equation}
Together with $-u>0$, this proves nonnegativity of the
first integral. The other two terms are nonnegative by
\eqref{eq:volume-curvature} and
\eqref{eq:dirichlet-lower}. Their sum is zero, so all three
vanish. Since the weight $-u$ is strictly positive in
$\Omega$, the first vanishing integral also gives
$S_{k+1,F}[u]=\binom{n}{k+1}$ almost everywhere.
Thus all equalities in \eqref{eq:all-deficits-zero} hold.

It remains to recover the domain and solution.
Equality in Lemma~\ref{lem:one-step-newton} makes every
eigenvalue of $A_F[u]$ equal to $1$ almost everywhere.
The matrix $A_F[u]$ is similar to the symmetric matrix
\eqref{eq:similar-matrix}, hence is diagonalizable.
Consequently,
\begin{equation}\label{eq:AF-identity}
A_F[u]=I\qquad\text{a.e. in }\Omega.
\end{equation}
Both $Du$ and $D_\xi E$ are Lipschitz, so $v$ is Lipschitz.
Its weak derivative is $I$: the identity holds off the
finite critical set, and that set has measure zero.
Hence $v(x)-x$ has zero weak derivative. It is constant
on every interior ball, and connectedness of $\Omega$
makes these constants agree. For some $x_0\in\R^n$,
\begin{equation}\label{eq:affine-gradient}
D_\xi E(Du(x))=x-x_0
\qquad(x\in\overline\Omega).
\end{equation}
The boundary assertion follows by continuity.
Polar duality \eqref{eq:gradient-inverse} gives
$Du(x)=DE_0(x-x_0)$, so connectedness again yields
$u(x)=E_0(x-x_0)+a$ for a constant $a$.
On $\partial\Omega$, \eqref{eq:norm-gradient-duality}
and $F(Du)=1$ imply $F_0(x-x_0)=1$.
The condition $u=0$ there therefore gives $a=-1/2$.

We have proved
\[
u(x)=\frac{F_0(x-x_0)^2-1}{2},\qquad
\Omega\subset x_0+\Wulff,\qquad
\partial\Omega\subset x_0+\partial\Wulff,
\]
where the inclusion of $\Omega$ uses $u<0$.
The relative boundary of $\Omega$ inside $x_0+\Wulff$
is empty. Thus $\Omega$ is both open and closed in the
connected set $x_0+\Wulff$; since it is nonempty,
$\Omega=x_0+\Wulff$.
Returning to the original variables multiplies the center
by $c$ and gives
$u(x)=(F_0(x-cx_0)^2-c^2)/2$.
Renaming $cx_0$ proves \eqref{eq:main-conclusion}.
\end{proof}
\subsection{Boundary formulations and scaling}
\noindent

\begin{proof}[Proof of Theorem~\ref{thm:neumann}]
The boundary equivalences \eqref{eq:N-equiv} give
$u_\nu=c/F(\nu)>0$ and $F(Du)=c$ on $\partial\Omega$.
Lemma~\ref{lem:N-adm}, with $\mu=\binom{n}{k}$, gives
$k$-admissibility. All the hypotheses of
Theorem~\ref{thm:main} are therefore satisfied.
That theorem and Lemma~\ref{lem:automatic-sign} give
\eqref{eq:main-conclusion} and $u<0$ in $\Omega$.
\end{proof}

\medskip
\begin{proof}[Proof of Corollary~\ref{cor:general-scaling}]
Let $\alpha$ and $R$ be given by \eqref{eq:alpha-R},
and set $w=u/\alpha$. Positive homogeneity gives
\[
A_F[w]=\alpha^{-1}A_F[u],\qquad
S_{k,F}[w]=\alpha^{-k}\mu=\binom{n}{k},
\qquad w=0\quad\text{on }\partial\Omega.
\]
The regularity and finite-critical-set assumptions are
preserved, and the eigenvalues of $A_F[w]$ are those of
$A_F[u]$ divided by the positive number $\alpha$.
In the boundary-gradient case, $w$ is therefore
$k$-admissible and $F(Dw)=c/\alpha=R$ on the boundary.
Theorem~\ref{thm:main} applies.
In the conormal case, the flux scales according to
\[
F(Dw)D_\xi F(Dw)\cdot\nu
=\alpha^{-1}F(Du)D_\xi F(Du)\cdot\nu
=RF(\nu),
\]
so Theorem~\ref{thm:neumann} gives the same conclusion
and establishes admissibility. In both cases,
\[
\Omega=x_0+R\Wulff,\qquad
w(x)=\frac{F_0(x-x_0)^2-R^2}{2}<0
\quad\text{in }\Omega.
\]
Multiplying by $\alpha>0$ gives
\eqref{eq:general-profile}, preserves the negative sign,
and preserves $k$-admissibility.
\end{proof}

\medskip
\begin{proof}[Proof of Corollary~\ref{cor:gradient-image}]
The almost-everywhere Hessian bound also holds in
distributions because $u\in C^{1,1}$.
Since $\Omega$ is convex,
$u-\theta|x|^2/2$ is convex on $\Omega$.
Its supporting-plane inequalities at $x$ and $y$,
added together, give
\[
(Du(x)-Du(y))\cdot(x-y)
\ge\theta|x-y|^2
\qquad(x,y\in\Omega).
\]
Continuity extends this inequality to
$x,y\in\overline\Omega$. In particular, $Du$ is injective
on the closed domain.

Fix $x\in\partial\Omega$ and approximate it by interior
points. Since $Du(\Omega)=B_F(c)$, continuity gives
$F(Du(x))\le c$.
If $F(Du(x))<c$, the prescribed gradient image supplies
$y\in\Omega$ with $Du(y)=Du(x)$, contradicting injectivity.
Thus $F(Du)=c$ on the entire boundary.
Moreover, $0\in B_F(c)=Du(\Omega)$ gives an interior
critical point, and injectivity makes it unique.
The finite-critical-set condition is therefore automatic.

At every noncritical point where the Hessian bound holds,
$B(Du)^{1/2}D^2uB(Du)^{1/2}$ is positive definite.
Its eigenvalues, and hence those of $A_F[u]$, belong to
$\Gamma_n\subset\Gamma_k$.
The single critical point has measure zero, so $u$ is
$k$-admissible. Theorem~\ref{thm:main} now proves the
corollary.
\end{proof}

\subsection{Deficit identity and equivalent characterizations}

\begin{proof}[Proof of Theorem~\ref{thm:deficit}]
The integral calculation in the proof of
Theorem~\ref{thm:main} derives
\eqref{eq:deficit-identity-intro} directly from
\eqref{eq:mean-u-exact}, \eqref{eq:Q-exact}, and
\eqref{eq:Q-energy-kplus}, before using the rigidity
conclusion. The three terms are nonnegative by
\eqref{eq:Skplus-upper}, \eqref{eq:volume-curvature},
and \eqref{eq:dirichlet-lower}, respectively.
Their sum is zero, so all three vanish.
Since $u<0$ in $\Omega$, the first vanishing integral
implies $S_{k+1,F}[u]=\binom{n}{k+1}$ almost everywhere.
This proves \eqref{eq:all-deficits-zero}.
For $k=1$, the curvature-volume equality already follows
from $\mathcal W_{0,F}(\Omega)=|\Omega|$.
\end{proof}

\medskip
\begin{proof}[Proof of Theorem~\ref{thm:equivalence}]
We prove the implications in a closed chain.
Assertion \textup{(i)} and the standing assumptions are
exactly the hypotheses of Theorem~\ref{thm:main}, so
\textup{(i)} implies \textup{(v)}.

For the model pair in \textup{(v)}, differentiation and
\eqref{eq:gradient-composition} give
$D_\xi E(Du)=x-x_0$, including at the center by continuity.
Thus $A_F[u]=I$ almost everywhere.
Also \eqref{eq:norm-gradient-duality} gives
$F(Du)=F_0(x-x_0)=c$ on the boundary.
Consequently, \textup{(v)} implies both \textup{(iv)}
and \textup{(i)}.
Assertion \textup{(iv)} implies \textup{(iii)} because
$S_{k+1}(I)=\binom{n}{k+1}$, and \textup{(iii)}
immediately implies \textup{(ii)}.

To obtain the converse, Lemmas~\ref{lem:one-step-newton}
and \ref{lem:automatic-sign} give
\[
\binom{n}{k+1}-S_{k+1,F}[u]\ge0,\qquad
-u>0\quad\text{in }\Omega.
\]
Under \textup{(ii)}, this nonnegative weighted integrand
has zero integral. The strict positivity of the weight
therefore gives \textup{(iii)}.
The equality statement in Lemma~\ref{lem:one-step-newton}
then makes every eigenvalue of $A_F[u]$ equal to $1$.
By its similarity to a symmetric matrix,
$A_F[u]$ is diagonalizable and equals $I$ almost
everywhere. Thus \textup{(iii)} implies \textup{(iv)}.

Finally, assume \textup{(iv)}.
The Lipschitz-flux argument leading to
\eqref{eq:affine-gradient} applies without any prescribed
boundary gradient. It gives
\[
D_\xi E(Du)=x-x_0,\qquad
u(x)=E_0(x-x_0)+a
\]
for some $x_0\in\R^n$ and a constant $a$.
Lemma~\ref{lem:automatic-sign} gives $u<0$ in $\Omega$.
Since $E_0\ge0$, this forces $a<0$; put
$c=\sqrt{-2a}>0$.
The zero boundary data then imply
$F_0(x-x_0)=c$ on $\partial\Omega$, whereas the interior
negative sign gives $F_0(x-x_0)<c$ in $\Omega$.
Hence
\[
\Omega\subset x_0+c\Wulff,\qquad
\partial\Omega\subset x_0+c\partial\Wulff.
\]
As in the proof of Theorem~\ref{thm:main}, the relative
boundary of $\Omega$ in the connected set
$x_0+c\Wulff$ is empty. It follows that
$\Omega=x_0+c\Wulff$, proving \textup{(v)}.
All five assertions are therefore equivalent.
\end{proof}

\section*{Declarations}

\noindent\textbf{Competing interests.} The authors declare no competing interests.

\medskip
\noindent\textbf{Data availability.} No datasets were generated or analyzed in this study.

\medskip
\noindent\textbf{Use of AI tools.}
The authors used AI tools to assist with literature searches,
language editing, and the formulation and checking of
mathematical arguments. The authors take full responsibility
for the correctness and integrity of the manuscript,
including all proofs.

\bigskip
\noindent\textsc{Lichun Liang}\\
School of Mathematical Sciences, Chongqing Normal University, Chongqing 401331,
China.\\
E-mail: \href{lianglichun@cqnu.edu.cn}{\textsf{lianglichun@cqnu.edu.cn}}

\medskip
\noindent\textsc{Yanjun Liu}\\
School of Mathematical Sciences, Chongqing Normal University, Chongqing 401331,
China.\\
E-mail: \href{mailto:liuyj@mail.nankai.edu.cn}{\textsf{liuyj@mail.nankai.edu.cn}}

\end{document}